\documentclass{elsart}
\usepackage{graphics}
\usepackage{graphicx}
\usepackage{epsfig}
\usepackage{amssymb,amsmath}
\usepackage[colorlinks]{hyperref}
\newtheorem{lemma}{Lemma}[section]

\newtheorem{theorem}{Theorem}[section]

\newtheorem{definition}{Definition}[section]
\newtheorem{remark}{Remark}[section]
\newtheorem{proof}{Proof}[section]
\begin{document}
\linespread{1}
\begin{frontmatter}
\title{Mixed multifractal densities for quasi-Ahlfors vector-valued measures}
\author{Adel Farhat}
\address{Labo. of Algebra, Number Theory and Nonlinear Analysis LR18ES50, Department of Mathematics, Faculty of Sciences, Monastir 5000, Tunisia.}
\ead{farhatadel1222@yahoo.fr}
\author{Anouar Ben Mabrouk\corauthref{cor1}\thanksref{label1}}
\address{Department of Mathematics, Higher Institute of Applied Mathematics and Computer Science, Street of Assad Ibn Alfourat, 3100 Kairouan, Tunisia.\\
Lab. Algebra, Number Theory and Nonlinear Analysis, UR11ES50, Department of Mathematics, Faculty of Sciences, 5000 Monastir, Tunisia.}
\ead{anouar.benmabrouk@fsm.rnu.tn}
\thanks[label1]{{Department of Mathematics, Faculty of Science, University of Tabuk, Saudi Arabia.}}
\corauth[cor1]{Corresponding author.}
\begin{abstract}
In the present work, some density estimations associated to vector-valued quasi-Ahlfors measures are developed within the mixed multifractal analysis framework. The principle idea reposes on the fact that being quasi-Ahlfors is sufficient to conduct a mixed multifractal analysis for vector-valued measures. In the present work, we introduced a multifractal density for finitely many measures, and showed that such density may be estimated well by means of the mixed multifractal measures. Such estimation induces an exact computation of multifractal spectrum of the vector-valued quasi-Ahlfors measure.
\end{abstract}
\begin{keyword}
Hausdorff measure and dimension, packing measure and dimension, Multifractal densities, Multifractal formalism, H\"olderian Measures, Ahlfors measures, Mixed cases.
\PACS: 28A78, 28A80.
\end{keyword}
\end{frontmatter}
\maketitle
\section{Introduction}
The purpose of the present paper is to develop a multifractal analysis for a special type of densities due to a special class of measures known as quasi-Ahlfors, within the framework of mixed multifractal analysis. The latter is a natural extension of multifractal analysis of single objects such as measures, functions, statistical data, distributions... It is developed quite recently (since 2014) in the pure mathematical point of view. In physics and statistics, it was appearing on different forms but not really and strongly linked to the mathematical theory. See for example \cite{Ganetal}, \cite{Maneveauetal}. In many applications such as clustering topics, each attribute in a data sample may be described by more than one type of measure. This leads researchers to apply measures well adopted for mixed-type data. See for example \cite{Ganetal}.

Mathematically speaking, the idea consists in establishing some control of the multifractal density of some vector-valued measures instead of a single measure as in the classic or original multifractal analysis of measures and to introduce a simultaneous density characteristic of such measures relatively to one of them which is characterized by a quasi-Ahlfors property. This is important as it permits for example to characterize fractal or irregular sets such as Moran ones. The present work will provide a natural extension of \cite{Attiaetal}, \cite{Bmabrouk1}, \cite{Bmabrouk3}, \cite{Cole}, \cite{Cole-Olsen}, \cite{Douzi-Selmi}, \cite{Menceuretal}.

The assumption of being Ahlfors for one of the measures is the essential motivation behind the present work, where we aim at a first step to review some existing works that have forgotten such assumption and developed some version of multifractal densities, even-though their constructions seems questionable (\cite{Attiaetal} and \cite{Cole}). The present work may form with \cite{Attiaetal}, \cite{Bmabrouk3}, \cite{Billingsley}, \cite{Cole}, \cite{Olsen1}, \cite{Olsen2} a quite full study of the concepts of multifractal densities of measures.

However, we noticed that there has been some lack in hypothesis in the last recent works \cite{Attiaetal} and \cite{Cole}. Although the developments in \cite{Attiaetal} are in some parts based on \cite{Cole} which also refers to \cite{Billingsley}, the authors did not pay attention to the fact that general probability measures (even-though being doubling) may not lead to multifractal dimensions. Indeed, it is already mentioned in \cite{Billingsley} that the lower bound in the definition of the $\mu$-Hausdorff measure due to a Borel probability measure $\mu$, (respectively, the upper bound in the definition of the $\mu$-packing measure) extends naturally on all the $\mu$-$\rho$-coverings of the analyzed sets (respectively, all the $\mu$-$\rho$-packings). Recall that, in \cite{Billingsley}, it is already mentioned that a $\mu$-$\rho$-covering imposes a control of the measure $\mu(C)$ of a cylinder by means of its radius $\rho$, such as $\mu(C)<\rho$. In \cite{Billingsley}, it is also excluded all atomic measures, as in this case we may not be able to find any $\mu$-$\rho$-coverings. Omitting these assumptions may lead to empty coverings (packings) suitable for inducing a multifractal analysis. 

To overcome these lacks, some weak hypothesis is assumed in the present work stating that one of the components of the vector-value measure should be quasi-Ahlfors. It consists of a weak form of the so-called Alhfors measures. See \cite{Edgar}, \cite{MattilaSaaramen}, \cite{Pajot} for backgrounds on these measures and their properties. In the sequel we will denote $\mathcal{P}_B(\mathbb{R}^d)$ the set of all Borel probability measures on $\mathbb{R}^d$, $n\geq1$. 
\begin{definition} A measure $\mu\in\mathcal{P}_B(\mathbb{R}^d)$ is said to be  quasi-Ahlfors regular with index $\alpha>0$ if 
$$
\displaystyle\limsup_{|U|\longrightarrow0}\displaystyle\frac{\mu(U)}{|U|^{\alpha}}<+\infty.
$$
\end{definition}
It is in fact easy to check that with this assumption, the extensions developed in \cite{Attiaetal} and \cite{Cole} become possible, and that an associated multifractal analysis may be addressed. 

For backgrounds and details on multifractal measures, dimensions, densities, examples, and couter-examples, the readers are asked to review \cite{Olsen1}, \cite{Olsen2}, \cite{Qu}, \cite{Xu-Wang}, \cite{Xu-Xu}, \cite{Xu-Xu-Zhong}, \cite{Ye1}, \cite{Ye1}, \cite{Yuan}, \cite{Zeng-Yuan-Xui}, \cite{Zhou-Feng}, \cite{Zhu-Zhou}).

The next section is devoted to some review on mixed multifractal analysis in general, and mixed multifractal densities in particular. Interactions with other fields, link with our results in the present work are addressed. Section 3 is devoted to the presentation of general mathematical settings and results that will be applied in the present study. In section 4, our main results in the present paper are provided without proofs. In section 5, proofs of main results are developed. Section 6 is a concluding part. Finally, section 7 is a brief appendix in which we provided the proof of a result applied in our study, which extends the well-known Billingsley result on dimension of measures. 
\section{Recent related works and motivations}
In a simple description, mixed multifractal analysis may be defined as mathematical glasses that permit to capture and/or to quantify the transient higher-order dependence beyond correlation of many measures, functions, time series, distributions, etc.

Mixed multifractal densities are applied in several contexts. In \cite{Jaerisch}, a mixed analysis of Markov maps has been conducted with at most countable number of branches. A mixed spectrum has been introduced known as mixed Birkhoff spectra combining many characteristics depending simultaneously on different sets of observations.

In \cite{Taylor}, a mixed analysis based idea has been applied in the context of multivariate time series with time-varying joint distribution. Based on the presence of possibly self-affine structure for the time series, the authors proposed a multivariate estimation of the distribution. Combrexelle in \cite{Combrexelle} showed a potential benefit of mixed multifractal analysis in remote sensing application. More precisely, the author addressed in \cite{Combrexelle} a multivariate analysis for images by developing a hierarchical models for distributions generalizing the Bayesian one. 

Hong et al (\cite{Hong}) developed a multifractal analysis of measures for the so-called mixed logical dynamical models to the classification of signals especially network traffic. These models are widely applied in the control of hybrid systems such as multi-server ones.

Mixed multifractal analysis of measures has been applied in \cite{Dai} to extract the properties of stock index series and in exploiting the eventual inner relationship that relates them. 

The last closest work to our's is developed in \cite{Khlifi}. However, the main difference is that the authors there considered the multifractal formalism for a measure relatively to another one controlled by the diameter of the covering elements. Instead of considering usual coverings, they applied the so-called $\nu$-$\delta$-coverings relatively to a suitable doubling measure $\nu$ to assure the control of both the measure of the covering and the diameters simultaneously, corresponding to the essential measure $\mu$ subject of the multifractal formalism. In the present work, we do not need such double covering assumption as quasi-Ahlfors assumption permits the multifractal analysis of the vector-valued measures and their densities. 

Besides, in \cite{DouziChaos2019}, the authors considered to revisit the well-known Besicovitch covering theorem in the framework of relative multifractal analysis. Results on multifractal densities in Tricot's sense have been established. The main idea consists in the use of doubling measures, which is not the case in the present work. In \cite{SelmiAMP2019}, density results have been established in the single case of multifractal analysis based on some necessary condition on strong regularity of measures. In our knowledge, the results developed in the present work are the first in the case of mixed multifractal analysis, for vector-valued measures and where weak assumption has been assumed on the vector. 

More about the use of mixed multifractal analysis of measures as well as functions and time series or images may be found in \cite{Abry,Kinnison} and the references therein. 
\section{General settings}
In this section, general settings will be reviewed especially those developed in \cite{Farhatetal}. We introduce the mixed multifractal generalisations of densities associated to vector-valued measures. Consider a vector $\mu=(\mu_{1},\mu_{2},...,\mu_{k})\in\left(\mathcal{P}_B(\mathbb{R}^d)\right)^k$ and denote
$$
\mu(B(x,r))=(\mu_{1}(B(x,r)),...,\mu_{k}(B(x,r))).
$$
For $q=(q_{1},q_{2},...,q_{k})\in\mathbb{R}^d$, we write
$$
\mu(B(x,r))^q=\left[\mu_{1}(B(x,r)\right]^{q_{1}}\times...\times \left[\mu_{k}(B(x,r)\right]^{q_{k}}.
$$
Let finally $\nu\in\mathcal{P}(\mathbb{R}^{d})$. The mixed multifractal generalisations of the Hausdorff measure with respect to $(\mu,\nu)$ is introduced in \cite{Farhatetal} as follows. For $E\subset\mathbb{R}^d$,
$$
{\overline{\mathcal{H}}}_{\mu,\nu,\delta}^{q,t}(E)=\inf\{\underset{i}{\sum}(\mu(B(x_{i},r_{i})))^{q}(\nu(B(x_{i},r_{i})))^{t}\}.
$$
The inf above is taken over the set of all centred $\delta$-coverings $(B(x_{i},r_{i})))_i$ of $E$. Next, we consider
$$
{\overline{\mathcal{H}}}_{\mu,\nu}^{q,t}(E)=\underset{\delta\downarrow0}{\text{lim}}{\overline{\mathcal{H}}}_{\mu,\nu,\delta}^{q,t}(E)=\underset{\delta>0}{\text{sup}}{\overline{\mathcal{H}}}_{\mu,\nu,\delta}^{q,t}(E),
$$
which induces finally,
$$
\mathcal{H}_{\mu,\nu}^{q,t}(E)=\underset{F\subseteq E}{\sup}{\overline{\mathcal{H}}}_{\mu,\nu}^{q,t}(F).
$$
Similarly to the mixed multifractal generalisation of Hausdorrf measure, the mixed generalized multifractal packing measure with respect to $(\mu,\nu)$ is introduced in \cite{Farhatetal} as follows. For $E\subset\mathbb{R}^d$,
$$
{\overline{\mathcal{P}}}_{\mu,\nu,\delta}^{q,t}(E)=\sup\{\;\underset{i}{\sum}(\mu(B(x_{i},r_{i})))^{q}(\nu(B(x_{i},r_{i}))^{t}\}.
$$
The sup above is taken over the set of all centred $\delta$-packings $(B(x_{i},r_{i})))_i$ of $E$. Let next
$$
{\overline{\mathcal{P}}}_{\mu,\nu}^{q,t}(E)=\underset{\delta\downarrow0}{\lim}{\overline{\mathcal{P}}}_{\mu,\nu,\delta}^{q,t}(E)=\underset{\delta>0}{\inf}{\overline{\mathcal{P}}}_{\mu,\nu,\delta}^{q,t}(E)
$$
and
$$
\mathcal{P}_{\mu,\nu}^{q,t}(E)=\underset{E\subseteq\;\underset{i}{\cup}E_{i}}{\inf}\underset{i}{\sum}{\overline{\mathcal{P}}}_{\mu,\nu}^{q,t}(E_{i}).
$$
The following Theorem proved in \cite{Farhatetal} resumes the characteristics of the mixed multifractal generalisations $\mathcal{H}_{\mu,\nu}^{q,t}$ and $\mathcal{P}_{\mu,\nu}^{q,t}$.
\begin{theorem}\label{Theorem1Part1}[\cite{Farhatetal}]
\begin{itemize}
\item The functions $\mathcal{H}_{\mu,\nu}^{q,t}$ and $\mathcal{P}_{\mu,\nu}^{q,t}$ are metric, outer measure and thus measures on Borel sets in $\mathbb{R}^{d}$.
\item For all $E\subset\mathbb{R}^d$, there exists unique extended reel numbers denoted $dim_{\mu,\nu}^{q}({E})$, $Dim_{\mu,\nu}^{q}({E}),\;\Delta_{\mu,\nu}^{q}({E})\in\left[-\infty,+\infty\right]$ satisfying respectively
$$
\mathcal{H}_{\mu,\nu}^{q,t}(E)=\left\{\begin{array}{lll}
\infty&\text{if}&t<dim_{\mu,\nu}^{q}({E}),\\
0&\text{if}&t>dim_{\mu,\nu}^{q}({E}),
\end{array}
\right.
$$
$$
\mathcal{P}_{\mu,\nu}^{q,t}(E)=\left\{
\begin{array}{lll}
\infty&\text{if}&t<Dim_{\mu,\nu}^{q}({E}),\\
0&\text{if}&t>Dim_{\mu,\nu}^{q}({E}),
\end{array}
\right.
$$
$$
{\overline{\mathcal{P}}}_{\mu,\nu}^{q,t}(E)=\left\{
\begin{array}{lll}
\infty&\text{if}&t<\Delta_{\mu,\nu}^{q}({E}),\\
0&\text{if}&t>\Delta_{\mu,\nu}^{q}({E}).
\end{array}
\right.
$$
\end{itemize}
\end{theorem}
For the convenience and to prove the necessity of the assumption of quasi-Ahlfors regularity mentioned above we recall in brief the proof of the first cutt-off value due to $\mathcal{H}_{\mu,\nu}^{q,t}$ in Theorem \ref{Theorem1Part1}. We claim firstly the following lemma.
\begin{lemma}
$\forall\,E\subseteq\mathbb{R}^{d}$ and $\forall q\in\mathbb{R}^{k}$, the set
$$
\Gamma_{q}=\left\{t\;\mathcal{H}_{\mu,\nu}^{q,t}(E)<+\infty\right\}\neq\emptyset.
$$
\end{lemma}
\begin{proof}
Let $M\in\mathbb{R}_{+}^{\ast}$ be such that
$$
\underset{\left|U\right|\rightarrow0}{{\overline{\lim}}}\frac{\nu(U)}{\left|U\right|^{\alpha}}<M.
$$
This implies that for some $\delta>0$, and $\forall r$; $0<r<\delta$, we have 
\begin{equation}
\nu(U)\leq M\left|U\right|^{\alpha};\forall\,U\hbox{ s.t. }\left|U\right|<r.
\end{equation}
Let next $(B(x_{i},r_{i}))_{i}$ be a $\delta$-covering of $E$ and consider the $N_B$ collections defined in the Besicovitch covering theorem. We get
\begin{equation}
\underset{i}{\sum}\mu(B(x_{i},r_{i}))^{q}\nu(B(x_{i},r_{i}))^{t}\leq\underset{i=1}{\overset{N_B}{\sum}}\underset{j}{\sum}\mu(B(x_{ij},r_{ij}))^{q}\nu(B(x_{ij},r_{ij}))^{t}.
\end{equation}
Whenever $q\geq 0,$ the right hand term is bounded by
$$
\underset{i=1}{\overset{N_B}{\sum}}\underset{j}{\sum}\nu(B(x_{ij},r_{ij}))^{t}.
$$
For $t=1,$ this becomes%
$$
\underset{i=1}{\overset{N_B}{\sum}}\underset{j}{\sum}\nu(B(x_{ij},r_{ij})).
$$
As the $(B(x_{ij},r_{ij}))_{j}$ are disjoint, the last quantity will be bounded by
$$
\underset{i=1}{\overset{N_B}{\sum}}\nu\left(\underset{j}{\cup}B(x_{ij},r_{ij})\right)\leq N_B\nu(\mathbb{R}^{d})=N_B.
$$
Consequently
$$
\mathcal{H}_{\mu,\nu}^{q,1}(E)<+\infty.
$$
Assume now that $q_{i}\leq0$ for some $i$, $1\leq i\leq k$. We get
$$
\underset{i}{\sum}\mu(B(x_{i},r_{i}))^{q}\nu(B(x_{i},r_{i}))^{t}\leq2^{-\alpha}M^{t}\underset{i}{\sum}\mu(B(x_{i},r_{i}))^{q}(2r_{i})^{\alpha t}.
$$
Let next $t>\frac{1}{\alpha}\left[\max\left(1,dim_{\mu}^{q}(E)\right)\right].$ We obtain
$$
\mathcal{H}_{\mu,\nu}^{q,t}(E)\leq2^{-\alpha t}M^{t}\mathcal{H}_{\mu}^{q,\alpha t}(E)<+\infty.
$$
\end{proof}
\begin{lemma}
i) $\mathcal{H}_{\mu,\nu}^{q,t}(E)<+\infty$ $\Rightarrow \mathcal{H}_{\mu,\nu}^{q,s}(E)=0,$ $\forall$ $s>t.$\\
ii) $\mathcal{H}_{\mu,\nu}^{q,t}(E)>0$ $\Rightarrow \mathcal{H}_{\mu,\nu}^{q,s}(E)=+\infty ,$ $\forall $ $s<t.$
\end{lemma}
\begin{proof}
i) Let $(B(x_{i},r_{i}))_{i}$ be a $\delta$-covering of $E$. We may write that
$$
\begin{array}{lll}
\displaystyle\sum_i\mu(B(x_{i},r_{i}))^q\nu(B(x_{i},r_{i}))^s
&=&\displaystyle\sum_i\mu(B(x_{i},r_{i}))^q\nu(B(x_{i},r_{i}))^{t}\nu(B(x_{i},r_{i}))^{s-t}\\
&\leq&M^{s-t}\delta^{s-t}\displaystyle\sum_i\mu(B(x_{i},r_{i}))^q\nu(B(x_{i},r_{i}))^{t}.
\end{array}
$$
Consequently,
\begin{equation}
{\overline{\mathcal{H}}}_{\mu,\nu,\delta}^{q,s}(E)\leq M^{s-t}\delta^{s-t}{\overline{\mathcal{H}}}_{\mu,\nu,\delta}^{q,t}(E).
\end{equation}
Hence,
$$
\mathcal{H}_{\mu,\nu}^{q,s}(E)=0.
$$
ii) Using the same arguments as in assertion i, we get
$$
{\overline{\mathcal{H}}}_{\mu,\nu,\delta}^{q,s}(E)\geq M^{s-t}\delta^{s-t}{\overline{\mathcal{H}}}_{\mu,\nu,\delta}^{q,t}(E)\hbox{ (as }s-t<0).
$$
Consequently,
$$
\mathcal{H}_{\mu,\nu}^{q,s}(E)=+\infty .
$$
\end{proof}
The two Lemmas above permit to introduce the generalised mixed multifractal Hausdorff dimension as
$$
dim_{\mu,\nu}^{q}(E)=\inf\left\{t,\,\mathcal{H}_{\mu,\nu}^{q,t}(E)=0\right\}=\sup\left\{t,\,\mathcal{H}_{\mu,\nu}^{q,t}(E)=+\infty \right\}.
$$
We now develop our idea about the mixed density of measures. The original definition are introduced in \cite{Cole} and \cite{Cole-Olsen}). It is next re-studied by several authors, especially in \cite{Attiaetal} and \cite{Olsen1}. In the present work, we conduct an extension of these works to the case introduced in \cite{Menceuretal} and their most recent generalizations in \cite{Farhatetal}.
\begin{definition}\label{mixeddensitiesdef}
Let $\theta\in\mathcal{P}(\mathbb{R}^{d})$, $x\in\,{Support}(\theta)$, $q=(q_1,q_2,...,q_k)\in\mathbb{R}^k$, and $t\in\mathbb{R}$. The upper, respectively, lower, $(q,t)$-density of $\theta$ at the point $x$ relatively to $(\mu,\nu)$ is defined by
$$
{\overline{d}}_{\mu,\nu}^{q,t}(x,\theta)=\underset{r\rightarrow0}{\lim\sup}\frac{\theta(B(x,r))}{(\mu(B(x,r)))^{q}(\nu(B(x,r))^{t}},
$$
respectively, 
$$
{\underline{d}}_{\mu,\nu}^{q,t}(x,\theta)=\underset{r\rightarrow 0}{\liminf}\frac{\theta(B(x,r))}{(\mu(B(x,r)))^{q}(\nu(B(x,r))^{t}}.
$$
Whenever ${\overline{d}}_{\mu,\nu}^{q,t}(x,\theta)={\underline{d}}_{\mu,\nu}^{q,t}(x,\theta)$, we denote by ${{d}}_{\mu,\nu}^{q,t}(x,\theta)$ their common value, which will be called the $(q,t)$-density of $\theta$ relatively to $(\mu,\nu)$ at the point $x$.
\end{definition}
Next, for $a>1$, and a single measure $\mu\in\mathcal{P}(\mathbb{R}^d)$, let
$$
P_{a}(\mu)=\underset{r\downarrow0}{\limsup}\left(\underset{x\in S_\mu}{\sup}\frac{\mu(B(x,ar))}{\mu(B(x,r))}\right),
$$
and for a vector-valued measure $\mu=(\mu_1,\mu_2,...,\mu_k)\in\bigl(\mathcal{P}(\mathbb{R}^d)\bigr)^k$,
$$
P_{a}(\mu)=\displaystyle\prod_{i=1}^kP_{a}(\mu_i).
$$
Finally, define the set of the so-called doubling vector-valued measures on $\mathbb{R}^{d},$ by
$$
P_{D}(\mathbb{R}^{d})=\displaystyle\bigcup_{a>1}\left\{\mu\in\mathcal{P}(\mathbb{R}^{d});\;P_{a}(\mu)<\infty\right\}.
$$
\section{Main results}
The first result of the present work deals with the establishment of lower and upper bounds for the mixed multifractal density introduced above. We will see that such bounds permit to obtain the multifractal formalism already introduced in \cite{Olsen1} and re-considered next in \cite{Attiaetal}, \cite{Cole-Olsen}, \cite{Farhatetal}, \cite{Menceuretal} and \cite{Menceuretal1}.
\begin{theorem}\label{mixeddensityestimatestheorem1}
There exists $C_1,C_2>0$ constants satisfying for all Borel set $E\subseteq\,S_\mu$, 
\begin{equation}\label{mixeddensityestimatestheorem1eq1}
C_1\mathcal{H}_{\mu,\nu}^{q,t}(E)\underset{x\in E}{\inf}{\overline{d}}_{\mu,\nu}^{q,t}(x,\theta)\leq\theta(E)\leq C_2\mathcal{H}_{\mu,\nu}^{q,t}(E)\underset{x\in E}{\sup}{\overline{d}}_{\mu,\nu}^{q,t}(x,\theta),
\end{equation}
whenever $\mathcal{H}_{\mu,\nu}^{q,t}(E)<\infty$, and
\begin{equation}\label{mixeddensityestimatestheorem1eq2}
C_1\mathcal{P}_{\mu,\nu}^{q,t}(E)\underset{x\in E}{\inf}{\underline{d}}_{\mu,\nu}^{q,t}(x,\theta)\leq\theta(E)\leq C_2\mathcal{P}_{\mu,\nu}^{q,t}(E)\underset{x\in E}{\sup}{\underline{d}}_{\mu,\nu}^{q,t}(x,\theta),
\end{equation}
whenever $\mathcal{P}_{\mu,\nu}^{q,t}(E)<\infty$.
\end{theorem}
\begin{remark}\label{remark1}
Whenever $\mu,\nu\in P_{D}(\mathbb{R}^{d})$, we may choose $C_1=C_2=1$ in Theorem \ref{mixeddensityestimatestheorem1}.
\end{remark}
As a result of the estimations above of the new mixed multifractal densities, we aim in the next step to show that such estimations permit in some special cases to compute the mixed multifractal spectrum for some suitable vector-valued measures in the mixed framework. For a Borel set $E\subset\mathbb{R}^{d}$, define
$$
{\overline{D}}_{\mu,\nu}^{q,t}(x,E)={\overline{d}}_{\mu,\nu}^{q,t}(x,\mathcal{H}_{\mu,\nu}^{q,s}(E))
$$
and
$$
{\underline{D}}_{\mu,\nu}^{q,t}(x,E)={\underline{d}}_{\mu,\nu}^{q,t}(x,\mathcal{H}_{\mu,\nu}^{q,s}(E)).
$$
Define similarly
$$
{\overline{\Delta}}_{\mu,\nu}^{q,t}(x,E)={\overline{d}}_{\mu,\nu}^{q,t}(x,\mathcal{P}_{\mu,\nu}^{q,t}(E))
$$
and
$$
{\underline{\Delta}}_{\mu,\nu}^{q,t}(x,E)={\underline{d}}_{\mu,\nu}^{q,t}(x,\mathcal{P}_{\mu,\nu}^{q,t}(E)).
$$
As usually, whenever
$$
{\overline{D}}_{\mu,\nu}^{q,t}(x,E)={\underline{D}}_{\mu,\nu}^{q,t}(x,E)
$$
we denote their common value by ${D}_{\mu,\nu}^{q,t}(x,E)$. Similarly, whenever
$$
{\overline{\Delta}}_{\mu,\nu}^{q,t}(x,E)={\underline{\Delta}}_{\mu,\nu}^{q,t}(x,E)
$$
we denote their common value by ${\Delta}_{\mu,\nu}^{q,t}(x,E)$. Denote next
$$
\underline{K}{=}\{\,x\in\,E,\;{\underline{D}}_{\mu,\nu}^{q,t}(x,E)=1\},\;\;
{\overline{K}}{=}\{\,x\in\,E,\;{\overline{D}}_{\mu,\nu}^{q,t}(x,E)=1\},
$$
$$
{\underline{T}}{=}\{\,x\in\,E,\;{\underline{\Delta}}_{\mu,\nu}^{q,t}(x,E)=1\},\;\;
{\overline{T}}{=}\{\,x\in\,E,\;{\overline{\Delta}}_{\mu,\nu}^{q,t}(x,E)=1\},
$$
$$
K={\underline{K}\cap\overline{K}}\;\text{\ and }\;T={\overline{T}\cap\underline{T}}.
$$
The multifractal dimensions of these sets are provided in the following theorem. 
\begin{theorem}\label{mixedmultifractaldimensions1}
Let $E\subset{S_\mu\cap S_\nu}$ be a Borel set.
\begin{enumerate}
\item If $\mathcal{H}_{\mu,\nu}^{q,t}(E)<\infty$ and $\mu,\nu\in P_{D}(\mathbb{R}^{d})$, then $dim_{\mu,\nu}^{q}({\overline{K}})=t$.
\item If $\mathcal{P}_{\mu,\nu}^{q,t}(E)<\infty $, then $Dim_{\mu,\nu}^{q}({\underline{T}})=t$.
\item If $\mathcal{P}_{\mu,\nu}^{q,t}(E)<\infty $ and $\mu,\nu\in P_{D}(\mathbb{R}^{d})$, we have the following equivalences.
\begin{description}
\item[a.] $\mathcal{H}_{\mu,\nu}^{q,t}=\mathcal{P}_{\mu,\nu}^{q,t}$.
\item[b.] ${\overline{D}}_{\mu,\nu}^{q,t}(x,E)={\underline{D}}_{\mu,\nu}^{q,t}(x,E)=1$, for $\mathcal{P}_{\mu,\nu}^{q,t}-a.a.x\in E$.
\item[c.] ${\overline{\Delta}}_{\mu,\nu}^{q,t}(x,E)={\underline{\Delta}}_{\mu,\nu}^{q,t}(x,E)=1$, for $\mathcal{P}_{\mu,\nu}^{q,t}-a.a$.$x\in E$.
\end{description}
\item If $\mathcal{H}_{\mu,\nu}^{q,t}=\mathcal{P}_{\mu,\nu}^{q,t}<\infty$ and $\mu,\nu\in P_{D}(\mathbb{R}^{d})$, then $dim_{\mu,\nu}^{q}({K})=dim_{\mu,\nu}^{q}({T})=t$.
\end{enumerate}
\end{theorem}
This result is important as it constitutes a first information leading to the computation of the multifractal spectrum due to the densities introduced. Indeed, related to the origins of the multifractal spectrum, such as in \cite{Bmabrouk1}, \cite{Bmabrouk3}, \cite{Cole}, \cite{Cole-Olsen}, \cite{Farhatetal}, \cite{Menceuretal}, \cite{Olsen1}, a starting point in the classical case is to establish an estimation of the form
$$
\theta(B(x,r))\sim(\mu(B(x,r))^{q}(2r)^{t\pm\varepsilon},\;\;r\rightarrow0,
$$
which by considering a somehow H\"older probability measure
$$
\nu(B(x,r))\sim(2r)^{t\pm\varepsilon},\;\;r\rightarrow0
$$
means that the densities considered above are all equal to 1. These last assumptions permit to compute the multifractal spectrum (evaluated as the Hausdorff dimension of the density level sets) by means of a Legendre transformation of a convex function issued from the multifractal generalized dimensions $b_{\mu,\nu}$, $B_{\mu,\nu}$ and $\Delta_{\mu,\nu}$.

In the following part, we aim to provide in a preparatory step a characterization of the sets of points in the support(s) of the relative measure(s) with the same density.
\begin{theorem}\label{theorem3}
Let $\mu$,$\nu $ $\in P_{D}(\mathbb{R}^{d})$ and $E\subset S_\mu\cap S_\nu$ be a Borel set. Let also
$$
E_{1}=\{\ x\in E,{\overline{D}}_{\mu,\nu}^{q,t}(x,E)={\underline{D}}_{\mu,\nu}^{q,t}(x,E)\}
$$
and
$$
E_{2}=\{\ x\in E,{\overline{\Delta}}_{\mu,\nu}^{q,t}(x,E)={\underline{\delta}}_{\mu,\nu}^{q,t}(x,E)\}.
$$
{1) If }$\mathcal{H}_{\mu,\nu}^{q,t}(E)<\infty $ then\\
a) ${\overline{D}}_{\mu,\nu}^{q,t}(x,E_{1})={\underline{D}}_{\mu,\nu}^{q,t}(x,E_{1}),$ for $\mathcal{H}_{\mu,\nu}^{q,t}-a.$ $e$ on $E_{1}.$\\
b) ${\overline{D}}_{\mu,\nu}^{q,t}(x,E\diagdown E_{1})={\underline{D}}_{\mu,\nu}^{q,t}(x,E\diagdown E_{1}),$ for $\mathcal{H}_{\mu,\nu}^{q,t}-a.$ $e$ on $E\diagdown E_{1}.$\\
2) {If }$\mathcal{P}_{\mu,\nu}^{q,t}(E)<\infty $ then\\
a) ${\overline{\Delta}}_{\mu,\nu}^{q,t}(x,E_{2})={\underline{\delta}}_{\mu,\nu}^{q,t}(x,E_{2}),$ for $\mathcal{P}_{\mu,\nu}^{q,t}-a.$ $e$ on $E_{2}.$\\
b) ${\overline{\delta}}_{\mu,\nu}^{q,t}(x,E\diagdown E_{2})={\underline{\delta}}_{\mu,\nu}^{q,t}(x,E\diagdown E_{2}),$ for $\mathcal{P}_{\mu,\nu}^{q,t}-a.$ $e$ on $E\diagdown E_{2}.$
\end{theorem}
\section{Proof of main results}
\subsection{Proof of Theorem \ref{mixeddensityestimatestheorem1}.}
We firstly show the left-hand side inequality of (\ref{mixeddensityestimatestheorem1eq1}). So, denote $\overline{d}=\underset{x\in E}{\inf}{\overline{d}}_{\mu,\nu}^{q,t}(x,\theta)$. Whenever $\overline{d}=0$, the inequality is obvious. So, assume that $\overline{d}>0$, and let $\eta$ be such that $0<\eta<\overline{d}$, and denote $\overline{d}_\eta=\overline{d}-\eta$. Let finally $F\subset E$ be closed, and $H\subset F$. Finally, for $\delta>0$, let
$$
B_{\delta}(F)=\{\ x\in\mathbb{R}^{d},\;dist(F,x)\leq\delta\}.
$$
It is straightforward that $B_{\delta}(F)\downarrow F$ whenever $\delta\downarrow0$. Therefore, for all $\varepsilon>0$, there is $\delta_{0}>0$ satisfying
$$
\theta(B_{\delta}(F))\leq\theta(F)+\varepsilon,\;\forall\,\delta,\;0<\delta<\delta_{0}.
$$
Besides, as $\mathcal{H}_{\mu,\nu}^{q,t}(H)<\infty$, we may also write
$$
{\overline{\mathcal{H}}}_{\mu,\nu}^{q,t}(H)-\varepsilon\leq{\overline{\mathcal{H}}}_{\mu,\nu,\delta}^{q,t}(H),\;\forall\,\delta,\;0<\delta<\delta_{0}.
$$
Denote next
$$
\Gamma_{\mu,\nu}^{q,t}(B(x,r))=[\mu(B(x,r))]^{q}[\nu(B(x,r))]^{t}
$$
and 
$$
\mathcal{B}_{\delta}=\left\{B(x,r);\,x\in H,\,0<r<\delta,\;\;\theta(B(x,r))\geq\overline{d}_\eta\Gamma_{\mu,\nu}^{q,t}(B(x,r))\right\}.
$$
From the definition of $\overline{d}$ and $\overline{d}_\eta$, there exists $\delta_0>0$ such that
$$
\displaystyle\frac{\left[\theta(B(x,r)\right]}{(\mu(B(x,r)))^{q}(\nu(B(x,r))^{t}}\geq\overline{d}_\eta;\;\;\forall r,\,0<r<\delta_0.
$$
Or equivalently,
$$
\theta(B(x,r))\geq\overline{d}_\eta\Gamma_{\mu,\nu}^{q,t}(B(x,r));\;\;\forall r,\,0<r<\delta_0.
$$
Henceforth, $\mathcal{B}_\delta\not=\emptyset$. So, let next $N_B$ be the number of at most countable collections $(\mathcal{B}_{i})_{1\leq i\leq N_B}=(B(x_{ij},r_{ij}))_{j,1\leq i\leq N_B}$ of $\mathcal{B}_\delta$ obtained from the Besicovitch covering theorem, where, for all $i$, $\mathcal{B}_{i}$ is composed of pairwise disjoint balls $B(x_{ij},r_{ij})$. We have,
$$
H\subset\bigcup_{1\leq i\leq N_B}\bigcup_jB(x_{ij},r_{ij}).
$$
It follows that
$$
\begin{array}{lll}
\displaystyle{\overline{\mathcal{H}}}_{\mu,\nu,\delta}^{q,t}(H) &\leq&{\overline{\mathcal{H}}}_{\mu,\nu,\delta}^{q,t}\,\bigl(\displaystyle\bigcup_{i=1}^{N_B}\bigcup_jB(x_{ij},r_{ij})\bigr)\\
&\leq&\displaystyle\sum_{i=1}^{N_B}{\overline{\mathcal{H}}}_{\mu,\nu,\delta}^{q,t}\,\bigl(\displaystyle\bigcup_jB(x_{ij},r_{ij})\bigr)\\
&\leq&\displaystyle\sum_{i=1}^{N_B}\sum_j[\mu(B(x_{ij},r_{ij}))]^{q}[\nu(B(x_{ij},r_{ij}))]^{q}\\
&\leq&\displaystyle\frac{1}{\overline{d}_\eta}\sum_{i=1}^{N_B}\sum_j\theta(B(x_{ij},r_{ij}))\\
&\leq&\displaystyle\frac{1}{\overline{d}_\eta}\sum_{i=1}^{N_B}\theta(B_\delta(F))\\
&\leq&\displaystyle\frac{N_B}{\overline{d}_\eta}\theta(B_\delta(F)).
\end{array}
$$
Consequently, we obtain
$$
\mathcal{H}_{\mu,\nu}^{q,t}(H)\leq{\overline{\mathcal{H}}}_{\mu,\nu,\delta}^{q,t}(H)+\varepsilon\leq\displaystyle\frac{N_B}{\overline{d}_\eta}\theta(B_\delta(F))+\varepsilon.
$$
As $\varepsilon\rightarrow0$, and observing that $\overline{d}_\eta=\overline{d}-\eta$, we get
$$
(\overline{d}-\eta)\mathcal{H}_{\mu,\nu}^{q,t}(H)\leq\,N\theta(B_\delta(F));\;\;\forall\eta>0,\;0<\eta<\overline{d}.
$$
By letting $\eta\downarrow0$, we get
$$
\overline{d}\mathcal{H}_{\mu,\nu}^{q,t}(H)\leq\,N_B\theta(B_\delta(F)).
$$
Now, for $\delta\downarrow0$, and taking the sup on $H\subset F$, we obtain
$$
\overline{d}\mathcal{H}_{\mu,\nu}^{q,t}(F)\leq\,N_B\theta(F).
$$
This is valid for all closed $F\subset E$. As a result, taking the sup on $F$, and replacing $\overline{d}$ by its exact form, we obtain
$$
\mathcal{H}_{\mu,\nu}^{q,t}(E)\underset{x\in E}{\inf}{\overline{d}}_{\mu,\nu}^{q,t}(x,\theta)\leq\,N_B\theta(E).
$$
We now proceed to show the right-hand side part of inequality (\ref{mixeddensityestimatestheorem1eq1}). As previously, let $\overline{D}=\displaystyle\sup_{x\in E}\overline{d}_{\mu,\nu}^{q,t}(x,\theta)$, $\eta>0$ and denote $\overline{D}+\eta$. For $\delta>0$, consider the set
$$
E_{\delta}=\{x\in E;\;\overline{D}_\eta\Gamma_{\mu,\nu}^{q,t}(B(x,r))\geq\theta(B(x,r),\;0<r<\delta\}.
$$
It is straightforward that for all $\delta>0$,
$$
{\overline{\mathcal{H}}}_{\mu,\nu}^{q,t}(E_{\delta})\leq\mathcal{H}_{\mu,\nu}^{q,t}(E_{\delta})<\infty.
$$
On the other hand, for all $\varepsilon>0$, let $(B(x_{i},r_{i}))_{i}$ be a $\delta$-covering of $E_{\delta}$ satisfying
$$
\displaystyle\sum_i[\mu(B(x_{i},r_{i}))]^q[\nu(B(x_{i},r))]^t\leq{\overline{\mathcal{H}}}_{\mu,\nu,\delta}^{q,t}(E_{\delta})+\varepsilon.
$$
We get
$$
\begin{array}{lll}
\theta(E_{\delta})&\leq&\theta(\bigcup_iB(x_{i},r_{i}))\\
&\leq&\displaystyle\sum_i\theta(B(x_{i},r_{i}))\\
&\leq&\overline{D}_\eta\displaystyle\sum_i[\mu(B(x_{i},r_{i}))]^q[\nu(B(x_{i},r))]^t\\
&\leq&\overline{D}_\eta\left[{\overline{\mathcal{H}}}_{\mu,\nu,\delta}^{q,t}(E_{\delta})+\varepsilon\right].
\end{array}
$$
As $E_\delta\subset E$, we get
$$
\theta(E_{\delta})\leq\overline{D}_\eta\left[\mathcal{H}_{\mu,\nu}^{q,t}(E)+\varepsilon\right].
$$
Whenever $\delta\downarrow0$, we obtain
$$
\theta(E)\leq\overline{D}_\eta\left[\mathcal{H}_{\mu,\nu}^{q,t}(E)+\varepsilon\right].
$$
This is true for all $\varepsilon,\eta>0$. Consequently,
$$
\theta(E)\leq\overline{D}\mathcal{H}_{\mu,\nu}^{q,t}(E),
$$
or equivalently
$$
\theta(E)\leq\displaystyle\sup_{x\in E}\overline{d}_{\mu,\nu}^{q,t}(x,\theta)\mathcal{H}_{\mu,\nu}^{q,t}(E).
$$
Now, we will prove the left-hand side of (\ref{mixeddensityestimatestheorem1eq2}). Denote similarly to the previous case $\underline{d}=\displaystyle\inf_{x\in E}{\underline{d}}_{\mu,\nu}^{q,t}(x,\theta)$, and assume here-also that $\underline{d}>0$. Let next $\eta$ be such that $0<\eta<\underline{d}$ and denote $\underline{d}_\eta=\underline{d}-\eta$. Let also $F\subset E$ be closed, and write for $\delta>0$,
$$
B_\delta(F)=\{x\in\mathbb{R}^{d};\; dist(F,x),\,\leq\delta\}.
$$
Let next for $\varepsilon >0$, $\delta_{0}>0$ be such that
$$
\theta(B_{\delta}(F))\leq\theta(F)+\varepsilon;\;\forall\delta,\,0<\delta<\delta_{0}.
$$
Denote next
$$
F_{\delta}=\left\{x\in F;\;\theta(B(x,r))\leq\underline{d}\Gamma_{\mu,\nu}^{q,t}(B(x,r)),\,0<r<\delta\right\}.
$$
Consider a $\delta$-packing $(B(x_{i},r_{i}))_{i}$ of $F_{\delta}$. It holds that
$$
\begin{array}{lll}
\underline{d}_\eta\displaystyle\sum_i[\mu(B(x_{i},r_{i}))]^q[\nu(B(x_{i},r_{i}))]^t
&\leq&\displaystyle\sum_i\theta(B(x_{i},r_{i}))\\
&\leq&\displaystyle\sum_i\theta(B(x_{i},r_{i}))\\
&\leq&\theta(\displaystyle\bigcup_iB(x_{i},r_{i}))\\
&\leq&\theta(B_{\delta}(F))\\
&\leq&\theta(F)+\varepsilon\\
&\leq&\theta(E)+\varepsilon,
\end{array}
$$
which yields that
$$
\underline{d}_\eta\mathcal{P}_{\mu,\nu}^{q,t}(F_{\delta})\leq \underline{d}_\eta{\overline{\mathcal{P}}}_{\mu,\nu}^{q,t}(F_{\delta})\leq\underline{d}_\eta{\overline{\mathcal{P}}}_{\mu,\nu,\delta}^{q,t}(F_{\delta})\leq\theta(E)+\varepsilon.
$$
Finally, letting $\delta$, $\varepsilon$ and $\eta$ $\rightarrow0$, we get
$$
\underline{d}\mathcal{P}_{\mu,\nu}^{q,t}(F)\leq\theta(E).
$$
It remains to check the inequality at the right-hand side of (\ref{mixeddensityestimatestheorem1eq2}). So, denote $\overline{D}=\displaystyle\sup_{x\in E}{\underline{d}}_{\mu,\nu}^{q,t}(x,\theta)$, $F\subset E$, and let $\varepsilon,\eta,\delta>0$ be such that
$$
{\overline{\mathcal{P}}}_{\mu,\nu,\delta}^{q,t}(F)\leq{\overline{\mathcal{P}}}_{\mu,\nu}^{q,t}(F)+\varepsilon.
$$
Denote next $\overline{D}_\eta=\overline{D}+\eta$, and
$$
\mathcal{B}_{\delta}=\left\{\,B(x,r);\;\overline{D}_\eta\Gamma_{\mu,\nu}^{q,t}(B(x,r))\leq\theta(B(x,r)),\,x\in F,\;0<r<\delta\right\}.
$$
Vitali's Theorem (\cite{Falconer}, \cite{Mattila}) implies that  
$$
\theta\left(F\setminus\left(\displaystyle\bigcup_iB(x_{i},r_{i})\right)\right)=0,
$$
for some $\delta$-packing $(B(x_{i},r_{i}))_{i}\subset\mathcal{B}_{\delta}$ of $F$. Furthermore, 
$$
\begin{array}{lll}
\theta(F)&=&\theta(\displaystyle\bigcup_i(F\cap B(x_{i},r_{i}))\\
&=&\theta((F\cap\displaystyle\bigcup_iB(x_{i},r_{i}))\\
&\leq&\displaystyle\sum_i\theta(F\cap B(x_{i},r_{i}))\\
&\leq&\displaystyle\sum_i\theta(B(x_{i},r_{i}))\\
&\leq&\overline{D}_\eta\displaystyle\sum_i[\mu(B(x_{i},r_{i}))]^q[\nu(B(x_{i},r_{i}))]^t\\
&\leq&\overline{D}_\eta{\overline{\mathcal{P}}}_{\mu,\nu,\delta}^{q,t}(F)\\
&\leq&\overline{D}_\eta\left[{\overline{\mathcal{P}}}_{\mu,\nu}^{q,t}(F)+\varepsilon\right].
\end{array}
$$
For $\varepsilon\rightarrow0$, and $\eta\rightarrow0$, we obtain
$$
\theta(F)\leq\overline{D}\,{\overline{\mathcal{P}}}_{\mu,\nu}^{q,t}(F),\;\;\forall F\subset E.
$$
And thus, the desired inequality follows.\vskip0.25cm
\subsection{Proof of Theorem \ref{mixedmultifractaldimensions1}.}
\textbf{1.} We shall prove that
\begin{equation}\label{proofthm2eq1}
0<\mathcal{H}_{\mu,\nu}^{q,t}(\overline{K})<\infty.
\end{equation}
Consider, in a first step, the set
$$
F=\{\ x\in E;{\overline{D}}_{\mu,\nu}^{q,t}(x,E)>1\},
$$
It may be seen as a limit of an increasing sequence of nested sets
$$
F_{m}=\{\ x\in E;{\overline{D}}_{\mu,\nu}^{q,t}(x,E)>1+\frac{1}{m}\},\;\;m\in\mathbb{N},
$$
as $F=$ $\underset{m}{\cup}F_{m}$. Hence, Theorem \ref{mixeddensityestimatestheorem1} (combined with Remark \ref{remark1}) yields that
$$
\mathcal{H}_{\mu,\nu}^{q,t}(F_{m})(1+\frac{1}{m})\leq\mathcal{H}_{\mu,\nu}^{q,t}(F_{m}),\;\forall\,m.
$$
As a result, we obtain
$$
\mathcal{H}_{\mu,\nu}^{q,t}(F_{m})=0,\text{ }\forall \text{ }m.
$$
Consequently,
$$
\mathcal{H}_{\mu,\nu}^{q,t}(F)=0,
$$
and thus
\begin{equation}\label{proofthm2eq2}
{\overline{D}}_{\mu,\nu}^{q,t}(x,E)\leq1,\;\text{ for }\mathcal{H}_{\mu,\nu}^{q,t}-a.\text{ }a.\text{ }x\in E\ .
\end{equation}
Next, consider similarly the set
$$
G\ =\{\ x\in E;{\overline{D}}_{\mu,\nu}^{q,t}(x,E)<1\ \},
$$
and analogously the sequence of nested sets
$$
G_{m}=\{x\in E;{\overline{D}}_{\mu,\nu}^{q,t}(x,E)\leq1-\frac{1}{m}\}.
$$
Again Theorem \ref{mixeddensityestimatestheorem1} (combined with Remark \ref{remark1}) yields that
$$
\mathcal{H}_{\mu,\nu}^{q,t}(G_{m})(1-\frac{1}{m})\geq\mathcal{H}_{\mu,\nu}^{q,t}(G_{m}),\text{ }\forall \text{ }m.
$$
Consequently,
$$
\mathcal{H}_{\mu,\nu}^{q,t}(G_{m})=0,\text{ }\forall \text{ }m.
$$
Since $G=\underset{m}{\cup}G_{m}$, we obtain
\begin{equation}\label{proofthm2eq3}
1\leq{\overline{D}}_{\mu,\nu}^{q,t}(x,E),\text{ for }\mathcal{H}_{\mu,\nu}^{q,t}-a.\text{ }a.\text{ }x\in E\ .
\end{equation}
Equations (\ref{proofthm2eq2}) and (\ref{proofthm2eq3}) yield that
$$
{\overline{D}}_{\mu,\nu}^{q,t}(x,E)=1,\text{ for }\mathcal{H}_{\mu,\nu}^{q,t}-a.\text{ }a.\text{ }x\in E.
$$
As a result,
$$
\mathcal{H}_{\mu,\nu}^{q,t}(\{\ x\in E,\text{ }{\overline{D}}_{\mu,\nu}^{q,t}(x,E)=1\})>0.
$$
Observing now that
$$
\mathcal{H}_{\mu,\nu}^{q,t}(\{\ x\in E,\text{ }{\overline{D}}_{\mu,\nu}^{q,t}(x,E)=1\})\leq\mathcal{H}_{\mu,\nu}^{q,t}(E)<\infty,
$$
it follows that
$$
0<\mathcal{H}_{\mu,\nu}^{q,t}(\{\ x\in E,\text{ }{\overline{D}}_{\mu,\nu}^{q,t}(x,E)=1\})<\infty.
$$
Hence (\ref{proofthm2eq1}) holds, and
$$
dim_{\mu,\nu}^{q}({\overline{K}})=t.
$$
Assertion \textbf{2.} may be checked by similar techniques.\vskip0.15cm\hskip-17pt
\textbf{3.} We shall prove that $a)\Rightarrow b)\Rightarrow c)\Rightarrow a)$.\\
$a)\Rightarrow b).$ From assertion \textbf{1.} above, it follows that
$$
{\overline{D}}_{\mu,\nu}^{q,t}(x,E)=1,\text{ for }\mathcal{H}_{\mu,\nu}^{q,t}-a.\text{ }a.\text{ }x\in E.
$$
So, as $\mathcal{H}_{\mu,\nu}^{q,t}=\mathcal{P}_{\mu,\nu}^{q,t}$, we get
\begin{equation}\label{aimplyiesbeq1}
{\overline{D}}_{\mu,\nu}^{q,t}(x,E)=1,\text{ for }\mathcal{P}_{\mu,\nu}^{q,t}-a.\text{ }a.\text{ }x\in E.
\end{equation}
Now, proceeding as in the proof of equations (\ref{proofthm2eq2}) and (\ref{proofthm2eq3}) in assertion \textbf{1.} above, we get
$$
{\underline{D}}_{\mu,\nu}^{q,t}(x,E)=1,\text{ for }\mathcal{P}_{\mu,\nu}^{q,t}-a.a.x\in E,
$$
which by the hypothesis $\mathcal{H}_{\mu,\nu}^{q,t}=\mathcal{P}_{\mu,\nu}^{q,t}$ yields that
$$
{\underline{D}}_{\mu,\nu}^{q,t}(x,E)=1,\text{ for }\mathcal{H}_{\mu,\nu}^{q,t}-a.a.x\in E.
$$
So, assertion \textbf{b.} is proved.\\
$b)\Rightarrow c).$ Using assertion \textbf{b.}, Theorem \ref{mixeddensityestimatestheorem1} (and Remark \ref{remark1}), it holds, for any ball $B(x,r)$, that
$$
\mathcal{H}_{\mu,\nu}^{q,t}(B(x,r))=\mathcal{P}_{\mu,\nu}^{q,t}(B(x,r)).
$$
Again, using assertion \textbf{b.} above, we get
$$
\mathcal{H}_{\mu,\nu}^{q,t}(B(x,r))\sim\Gamma_{\mu,\nu}^{q,t}(B(x,r)),\;\;\hbox{for}\;\;\mathcal{H}_{\mu,\nu}^{q,t}-a.a.x\in E.
$$
Consequently,
$$
\mathcal{P}_{\mu,\nu}^{q,t}(B(x,r))\sim\Gamma_{\mu,\nu}^{q,t}(B(x,r)),\;\;\hbox{for}\;\;\mathcal{P}_{\mu,\nu}^{q,t}-a.a.x\in E.
$$
Hence,
$$
{\underline{\Delta}}_{\mu,\nu}^{q,t}(x,E)={\overline{\Delta}}_{\mu,\nu}^{q,t}(x,E)=1,\text{ for }\mathcal{P}_{\mu,\nu}^{q,t}-a.a.x\in E.
$$
So as assertion \textbf{c.}\\
$c)\Rightarrow a).$ Applying Theorem \ref{mixeddensityestimatestheorem1} and Remark \ref{remark1}, we get for all Borel set $F\subset E$,
$$
\mathcal{H}_{\mu,\nu}^{q,t}(F)\leq\mathcal{H}_{\mu,\nu}^{q,t}(F)\leq\mathcal{P}_{\mu,\nu}^{q,t}(F).
$$
Consequently,
$$
\mathcal{P}_{\mu,\nu}^{q,t}=\mathcal{H}_{\mu,\nu}^{q,t}.
$$
4) is an immediate consequence of assertions \textbf{1.}, \textbf{2.} and \textbf{3.}
\subsection{Proof of Theorem \ref{theorem3}}
We develop here-after the proof of assertion \textbf{1.a.} The remaining assertions may be proved by following similar techniques.\\
\textbf{1.a.} We claim that for all $\mathcal{H}_{\mu,\nu}^{q,t}$-measurable set $F\subset E$, we have
\begin{equation}\label{proofthm3eq1}
{\overline{D}}_{\mu,\nu}^{q,t}(x,E)={\overline{D}}_{\mu,\nu}^{q,t}(x,F)\text{ and }{\underline{D}}_{\mu,\nu}^{q,t}(x,E)={\underline{D}}_{\mu,\nu}^{q,t}(x,F);\,\mathcal{H}_{\mu,\nu}^{q,t}-a.e\text{ on }F.
\end{equation}
Indeed, denote as in \cite{Attiaetal}, \cite{Cole-Olsen}, for $\theta\in P(\mathbb{R}^{d})$,
$$
\theta_{E}(A)=\theta(E\cap A)\;\mbox{and}\;\lambda_E(A)=\theta(A\cap E^c),\;\text{ for all Borel set }A.
$$
It is straightforward that
$$
{\overline{d}}_{\mu,\nu}^{q,t}(x,\theta_{E})\leq{\overline{d}}_{\mu,\nu}^{q,t}(x,\theta)
$$
and
$$
{\underline{d}}_{\mu,\nu}^{q,t}(x,\theta_{E})\leq{\underline{d}}_{\mu,\nu}^{q,t}(x,\theta).
$$
On the other hand, we have
$$
{\underline{d}}_{\mu,\nu}^{q,t}(x,\theta)
\leq{\underline{d}}_{\mu,\nu}^{q,t}(x,\theta_{E})+{\overline{d}}_{\mu,\nu}^{q,t}(x,\lambda_E)
$$
and
$$
{\overline{d}}_{\mu,\nu}^{q,t}(x,\theta)
\leq{\overline{d}}_{\mu,\nu}^{q,t}(x,\theta_{E})+{\overline{d}}_{\mu,\nu}^{q,t}(x,\lambda_E).
$$
We now claim that
\begin{equation}\label{proofthm3eq2}
{\overline{d}}_{\mu,\nu}^{q,t}(x,\lambda_E)=0.
\end{equation}
Indeed, the set $G=\{x;\;{\overline{d}}_{\mu,\nu}^{q,t}(x,\lambda_E)\not=0\}$ is a countable union of
$$
G_{k}=\{\ x\in E;\text{ }{\overline{d}}_{\mu,\nu}^{q,t}(x,\lambda )\geq\frac{1}{k}\},\;\;k\geq1.
$$
From Theorem \ref{mixeddensityestimatestheorem1} (and Remark \ref{remark1}), we get
$$
\lambda_E(G_{k})\geq\frac{1}{k}\mathcal{H}_{\mu,\nu}^{q,t}(G_{k}).
$$
Consequently,
$$
\mathcal{H}_{\mu,\nu}^{q,t}(G_{k})=0,\;\forall k,
$$
and thus
$$
\mathcal{H}_{\mu,\nu}^{q,t}(G)=0.
$$
Therefore,
$$
{\overline{d}}_{\mu,\nu}^{q,t}(x,\lambda_E)=0\text{ for }\mathcal{H}_{\mu,\nu}^{q,t}-a.\text{ }a\text{ on }E,
$$
which leads to (\ref{proofthm3eq2}). As a result, we get
$$
{\underline{d}}_{\mu,\nu}^{q,t}(x,\theta)
\leq{\underline{d}}_{\mu,\nu}^{q,t}(x,\theta_{E})
$$
and
$$
{\overline{d}}_{\mu,\nu}^{q,t}(x,\theta)
\leq{\overline{d}}_{\mu,\nu}^{q,t}(x,\theta_{E}).
$$
These estimations together yield claim (\ref{proofthm3eq1}).
\section{Conclusion}
In the present paper, a class of quasi Ahlfors vector-valued measures has been applied for the estimation of mixed multifractal densities. We served of a mixed multifractal analysis developed relatively to such a class to show that the computation of the fractal dimension is always possible when estimating well the mixed multifractal densities. Even-though the present work is more theoretical, however, many practical cases may be addressed. Indeed, such densities have been applied for many time in the understanding of many phenomena, such as climate factors distributions in different regions, crashes propagation and the effect on markets, etc. See \cite{Abry,Chen,Figueiredo,Garcia,Pavon,Scheuring}. These applications have been tackled by using the single multifractal analysis and densities. However, in the recent decades there has been a 'proof' of the fact that taking simultaneously many phenomena and/or many regions for the same task may induce well understanding. This may be investigated by means of the mixed analysis.  

In a pure mathematical point of view, a new framework of mixed multifractal analysis for vector-valued measures has been developed in \cite{Menceuretal1}, where a general context has been considered, in which the measures considered are not necessary Gibbs, nor doubling, nor quasi-Ahlfors, but controlled by an extra gauge function to prove the validity of the multifractal formalism. An eventual extension of the present work is to combine with \cite{Menceuretal1} for more general classes of multifractal densities. 
\section{Appendix - The quasi-Ahlfors regularity}
Remark that whenever the measure $\nu$ is quasi-Ahlfors regular with index $\alpha$, we get for $t\geq0$,
$$
\mathcal{H}_{\mu,\nu}^{q,t}\leq M^t\mathcal{H}_{\mu}^{q,\alpha t}.
$$
Henceforth, for $\alpha t>max(0,dim_\mu^q)$, we get
$$
t\geq dim_{\mu,\nu}^q,
$$
which yields that
$$
\left(dim_\mu^q\right)_+\geq\alpha\,dim_{\mu,\nu}^q.
$$
Similarly, we get
$$
\left(dim_\mu^q\right)_-\leq\alpha\,dim_{\mu,\nu}^q.
$$
As a result, we obtain an analogous of Billingsley theorem as
\begin{equation}\label{analogueBillingsley}
\left(dim_\mu^q\right)_-\leq\alpha\,dim_{\mu,\nu}^q\leq
\left(dim_\mu^q\right)_+.
\end{equation}
Whenever $\nu$ is $\alpha$-Ahlfors regular we get a mixed version of Billingsley theorem
\begin{equation}\label{BillingsleyMixed}
dim_\mu^q=\alpha\,dim_{\mu,\nu}^q.
\end{equation}

\end{document}